\documentclass{amsart}

\usepackage{lmodern}
\usepackage[utf8]{inputenc}
\usepackage[english]{babel}

\usepackage{fullpage}

\usepackage{amsmath}
\usepackage{amsfonts}
\usepackage{amsthm}
\usepackage{amssymb}
\usepackage{theoremref}
\usepackage{colonequals}
\usepackage{enumitem}
\usepackage{aligned-overset}

\newcounter{master}
\numberwithin{master}{section}

\theoremstyle{plain}

\newtheorem{theorem}[master]{Theorem}
\newtheorem{proposition}[master]{Proposition}
\newtheorem{lemma}[master]{Lemma}

\newtheorem*{question}{Question}

\theoremstyle{definition}

\newtheorem*{definition}{Definition}
\newtheorem{claim}[master]{Claim}
\newtheorem*{notation}{Notation}

\theoremstyle{remark}

\newtheorem*{remark}{Remark}

\makeatletter
\let\c@equation\c@master

\let\c@figure\c@master

\let\c@table\c@master

\makeatother

\makeatletter
\def\@cite#1#2{[\textbf{#1}\if@tempswa , #2\fi]}
\def\@biblabel#1{[\textbf{#1}]}
\makeatother

\begin{document}

\title{On separably integrable symmetric convex bodies}

\author[V. Yaskin]{Vladyslav Yaskin}
\address{University of Alberta, Department of Mathematical and Statistical Sciences}
\email{vladyaskin@math.ualberta.ca}
\author[B. Zawalski]{Bartłomiej Zawalski}
\address{Polish Academy of Sciences, Institute of Mathematics}
\email{b.zawalski@impan.pl}
\thanks{The first author was supported in part by NSERC. The second author was supported by the Polish National Science Centre grant number 2020/02/Y/ST1/00072.}

\begin{abstract}
An infinitely smooth symmetric convex body $K\subset\mathbb R^d$ is called $k$-separably integrable, $1\leq k<d$, if its $k$-dimensional isotropic volume function $V_{K,H}(t)=\mathcal H^d(\{\boldsymbol x\in K:\mathrm{dist}(\boldsymbol x,H^\perp)\leq t\})$ can be written as a finite sum of products in which the dependence on $H\in\mathrm{Gr}(k,\mathbb R^d)$ and $t\in\mathbb R$ is separated. In this paper, we will obtain a complete classification of such bodies. Namely, we will prove that if $d-k$ is even, then $K$ is an ellipsoid, and if $d-k$ is odd, then $K$ is a Euclidean ball. This generalizes the recent classification of polynomially integrable convex bodies in the symmetric case.
\end{abstract}

\maketitle
\tableofcontents

\section{Introduction}\label{sec:01}

Newton argued in \emph{Principia} that the areas of caps of a planar convex body with infinitely smooth boundary are not expressible in terms of algebraic equations:

\begin{theorem}[{\cite[\S VI, Lemma XXVIII]{newton1723philosophiae}}]
There is no oval figure whose area, cut off by right lines at pleasure, can be universally found by means of equations of any number of finite terms and dimensions.
\end{theorem}

On the other hand, it was already known since the remarkable result of Archimedes that the volume cut off by a plane from a Euclidean ball in $\mathbb R^3$ depends algebraically on the plane. Further, it can be easily verified that the latter is true for any ellipsoid in any odd-dimensional space. It is, therefore, natural to ask to what extent can Newton's result be generalized.\\

Several related questions of this type were brought up by Arnold \cite[1987-14, 1988-13, 1990-27]{arnold2004arnold} in his famous seminar at Moscow State University. In 2015, Vassiliev \cite{https://doi.org/10.1112/blms/bdv002} solved the problem 1988-13 by showing that if $K\subset\mathbb R^d$ is a bounded domain in an even-dimensional space, then the volumes $V_{K,H}^\pm(t)$ cut off by a hyperplane parallel to $H\in\mathrm{Gr}(d-1,\mathbb R^d)$ at distance $ t\in\mathbb R$ from the origin are not algebraic functions of $H$ and $t$. Two years later, Agranovsky \cite{agranovsky2018complex} suggested a new direction for further research, introducing the concept of polynomial integrability.

\begin{definition}[{\cite[\S 2.1]{koldobsky2005fourier}}]
Let $K\subset\mathbb R^d$ be a convex body. For $\boldsymbol\xi\in\mathbb S^{d-1}$ we define the \emph{parallel section function} of $K$ by
$$A_{K,\boldsymbol\xi}(t)\colonequals\mathcal H^{d-1}(K\cap\{\langle\boldsymbol\xi\rangle^\perp+t\boldsymbol\xi\}),$$
where $\{\langle\boldsymbol\xi\rangle^\perp+t\boldsymbol\xi\}$ is the hyperplane perpendicular to $\boldsymbol\xi$ at distance $t$ from the origin.
\end{definition}

\begin{definition}[{\cite[Definition~1.1]{KOLDOBSKY2017876}}]
Let $K$ be a convex body in $\mathbb R^d$. Then $K$ is called \emph{polynomially integrable} if its parallel section function $A_{K,\boldsymbol\xi}(t)$ is a polynomial in $t$ on its support, i.e.,
$$A_{K,\boldsymbol\xi}(t)=\sum_{i=0}^na_i(\boldsymbol\xi)t^i,\quad a_i:\mathbb S^{d-1}\to\mathbb R.$$
\end{definition}

Agranovsky showed that if $K\subset\mathbb R^d$ is a domain with a smooth boundary in an even-dimensional space (the smoothness assumption was already known to be necessary), then it is not polynomially integrable \cite[Theorem~2]{agranovsky2018complex}. Equivalently, the volume $V_{K,H}^\pm(t)$ is not a polynomial in $t$. Since, on the one hand, the assumption of polynomial integrability imposes additional constraints on the dependence of $V_{K,H}^\pm(t)$ on $t$ but, on the other, removes all constraints on the dependence of $V_{K,H}^\pm(t)$ on $H$, his result has a slightly different flavor. Agranovsky also showed that in an odd-dimensional space, all polynomially integrable bodies must be convex \cite[Theorem~5]{agranovsky2018complex}. Their classification was completed shortly afterward by Koldobsky, Merkurjev, and Yaskin \cite{KOLDOBSKY2017876}, who proved the following theorem:

\begin{theorem}[{\cite[Theorem~3.7]{KOLDOBSKY2017876}}]\thlabel{thm:02}
Let $d$ be an odd positive integer. If $K$ is an infinitely smooth polynomially integrable convex body in $\mathbb R^d$, then $K$ is an ellipsoid.
\end{theorem}

\noindent In a recent work, Agranovsky, Koldobsky, Ryabogin, and Yaskin \cite{AGRANOVSKY2023127071} established a similar result in even-dimensional spaces, assuming that the parallel section function $A_{K,\boldsymbol\xi}(t)$ can be expressed in the form
$$A_{K,\boldsymbol\xi}(t)=P(\boldsymbol\xi,t)\sqrt{Q(\boldsymbol\xi,t)},$$
where $P,Q$ are polynomials in $t$ and $\deg Q=2$ \cite[Theorem~1.4]{AGRANOVSKY2023127071}.\\

For other developments on the problem, the reader is referred to \cite{inbook,Agranovsky2020,Agranovsky2022,Boman2021,Boman+2021+351+367,article1,Vassiliev2019,Vassiliev2020}.

\section{Statement of the result}

In this paper, we are going to significantly weaken the polynomial integrability condition. Before we do this, however, we need to introduce the notion of separable integrability.

\begin{definition}
Suppose that $X,Y$ are Hausdorff spaces and $A,B$ are subalgebras of $C(X,\mathbb R),C(Y,\mathbb R)$, respectively. The \emph{algebraic tensor product} $A\otimes B$ of $A,B$ is a subalgebra of $C(X\times Y,\mathbb R)$ generated by \emph{pure products} of the form
$$(a\otimes b)(x,y)\colonequals a(x)b(y),\quad a\in A,\ b\in B.$$
\end{definition}

\begin{definition}
A function in $C(X\times Y,\mathbb R)$ is called \emph{separable} if it is a finite sum of pure products (i.e., an element of the algebraic tensor product $C(X,\mathbb R)\otimes C(Y,\mathbb R)$). A function is called \emph{entangled} if it is not separable.
\end{definition}

Denote by $\mathrm{Gr}(k,\mathbb R^d)$ the \emph{Grassmann manifold} of all $k$-dimensional linear subspaces of $\mathbb R^d$.

\begin{definition}
Let $K\subset\mathbb R^d$ be a convex body. For a $k$-dimensional linear subspace $H\in\mathrm{Gr}(k,\mathbb R^d)$, $1\leq k<d$, we define the \emph{$k$-dimensional isotropic volume function} of $K$ by
$$V_{K,H}(t)=\mathcal H^d(\{\boldsymbol x\in K:\mathrm{dist}(\boldsymbol x,H^\perp)\leq t\}),$$
where $\{\boldsymbol x\in K:\mathrm{dist}(\boldsymbol x,H^\perp)\leq t\}$ is the intersection of $K$ with a $k$-dimensional right circular hypercylinder with base space $H$, axis $H^\perp$ and radius $t$.
\end{definition}

Being careful in making delicate decisions about the domain, we will intentionally define the key concept of separable integrability only locally.

\begin{definition}
A convex body $K\subset\mathbb R^d$ is called \emph{locally $k$-separably integrable}, $1\leq k<d$, if its $k$-dimensional isotropic volume function $V_{K,H}(t):\mathrm{Gr}(k,\mathbb R^d)\times[0,\infty)\to\mathbb R$ is separable in some open neighborhood $U$ of $\mathrm{Gr}(k,\mathbb R^d)\times\{0\}$, i.e.,
\begin{equation}\label{eq:01}V_{K,H}(t)=\sum_{i=0}^na_i(H)b_i(t),\quad a_i:\mathrm{Gr}(k,\mathbb R^d)\to\mathbb R,\ b_i:[0,\infty)\to\mathbb R\end{equation}
for every $(H,t)\in U$.
\end{definition}

\begin{remark}
By definition, if a convex body $K\subseteq\mathbb R^d$ is polynomially integrable, then it is also locally $1$-separably integrable. Indeed, for $t$ such that the interval $[-t,+t]$ is contained in the support of $A_{K,\boldsymbol\xi}$ we have
$$V_{K,\langle\boldsymbol\xi\rangle}(t)=\int_{[-t,+t]}A_{K,\boldsymbol\xi}(r)\;\mathrm dr=\sum_{i=0}^na_i(\boldsymbol\xi)\left(\int_{[-t,+t]}r^i\;\mathrm dr\right).$$
\end{remark}

\begin{remark}
If $d-k$ is even and $K\subseteq\mathbb R^d$ is an ellipsoid, then it is locally $k$-separably integrable. Indeed, for $t$ such that the ball $\mathbb B^d(t)\cap H$ is contained in the projection $K\mid H$ we have
\begin{align*}
V_{K,H}(t)&=\int_{\mathbb B^d(t)\cap H}\mathcal H^{d-k}(K\cap\{H^\perp+\boldsymbol u\})\;\mathrm d\boldsymbol u\\
&=\int_{\mathbb S^{d-1}\cap H}\int_{[0,t]}r^{k-1}\mathcal H^{d-k}(K\cap\{H^\perp+r\boldsymbol\theta\})\;\mathrm dr\;\mathrm d\boldsymbol\theta\\
&=\int_{\mathbb S^{d-1}\cap H}\int_{[0,t]}r^{k-1}A_{K\cap\langle H^\perp,\boldsymbol\theta\rangle,\boldsymbol\theta}(r)\;\mathrm dr\;\mathrm d\boldsymbol\theta.
\end{align*}
Now, since $K\cap\langle H^\perp,\boldsymbol\theta\rangle$ is an ellipsoid in an odd-dimensional space $\langle H^\perp,\boldsymbol\theta\rangle$, it is polynomially integrable. It follows that
$$V_{K,H}(t)=\int_{\mathbb S^{d-1}\cap H}\int_{[0,t]}r^{k-1}\sum_{i=0}^{d-k}a_{H,i}(\boldsymbol\theta)r^i\;\mathrm dr\;\mathrm d\boldsymbol\theta=\sum_{i=0}^{d-k}\left(\int_{\mathbb S^{d-1}\cap H}a_{H,i}(\boldsymbol\theta)\;\mathrm d\boldsymbol\theta\right)\left(\int_{[0,t]}r^{k+i-1}\;\mathrm dr\right).$$
\end{remark}

Our main result is the following theorem:

\begin{theorem}\thlabel{thm:01}
Let $K\subseteq\mathbb R^d$ be an origin-symmetric convex body with infinitely smooth boundary $\partial K$. If $K$ is locally $k$-separably integrable, then $d-k$ is even and $K$ is an ellipsoid or $d-k$ is odd and $K$ is a Euclidean ball.
\end{theorem}

Note that none of the results mentioned in \S\ref{sec:01} requires $K$ to be symmetric. Therefore \thref{thm:01} generalizes \thref{thm:02} only under this additional assumption. Unfortunately, exactly as in \cite{KOLDOBSKY2017876}, the non-symmetric case is essentially more difficult and requires more involved algebraic arguments. Nevertheless, \thref{thm:01} seems to indicate the crux of polynomial integrability. Namely, it is not so much the rigidity of polynomials that makes \thref{thm:02} true as the fact that the linear space of polynomials of fixed degree is finite-dimensional. Interestingly enough, \cite[Theorem~1.4]{AGRANOVSKY2023127071} generalizes \thref{thm:02} in a completely different way than \thref{thm:01}. On the one hand, it still needs the rigidity of polynomials, but on the other, it is more flexible in terms of the linear structure. This phenomenon prompts us to ask the following question:

\begin{question}
Let $K$ be a bounded domain in $\mathbb R^d$ with an infinitely smooth boundary $\partial K$. If the $k$-dimensional isotropic volume function $V_{K,H}(t)$ can be locally expressed in the form
$$V_{K,H}(t)=\Phi(a_1(H),a_2(H),\ldots,a_m(H),b_1(t),b_2(t),\ldots,b_n(t))$$
on some open neighborhood of $\mathrm{Gr}(k,\mathbb R^d)\times\{0\}$, where $\Phi:\mathbb R^{m+n}\to\mathbb R$ is algebraic and $a_i:\mathrm{Gr}(k,\mathbb R^d)\to\mathbb R$, $b_i:[0,\infty)\to\mathbb R$ are smooth, is $K$ necessarily an ellipsoid?
\end{question}

\noindent It contains all the aforementioned results, including ours. To the authors' best knowledge, no counterexample is known so far.

\begin{remark}
By the superposition theorem of Kolmogorov \cite{MR0111809}, there always exist monotonic increasing functions $a_i\in C(\mathrm{Gr}(k,\mathbb R^d),\mathbb R)$, $b_i\in C([0,\infty),\mathbb R)$ with the property that each continuous function $V_{K,H}(t)\in C(\mathrm{Gr}(k,\mathbb R^d)\times[0,\infty),\mathbb R)$ can be (locally) represented in the form
$$V_{K,H}(t)=\sum_{i=1}^5\phi_i(a_i(H)+b_i(t))$$
with functions $\phi_i\in C(\mathbb R,\mathbb R)$. Therefore the question is not really about the separability of variables or even finiteness of the representation, but rather if the individual functions in such a representation can be made infinitely smooth or even algebraic. This type of question is already much more delicate, as in Kolmogorov's proof, there is an overt rivalry between the smoothness of $a_i$ and $b_i$ and the smoothness of $\phi_i$. This also indicates why the initial smoothness assumption was crucial. However, since we do not insist that all the functions $\phi_i$ should be one-parameter, we seem to avoid the basic difficulty (cf. \cite{MR0165138}). After all, our question may be considered as yet another (local) variant of the superposition problem for a particular class of functions arising as $k$-dimensional isotropic volumes of smooth convex bodies.
\end{remark}

\section{Definitions and basic concepts}

We will begin with a brief reminder of the basic concepts and definitions that we will frequently use in the rest of the work.

\begin{notation}
Throughout the text, we will use the multi-index notation. A $d$-dimensional \emph{multi-index} is a $d$-tuple $\boldsymbol\alpha=(\alpha_1,\alpha_2,\ldots,\alpha_d)$ of non-negative integers. For multi-indices $\boldsymbol\alpha,\boldsymbol\beta\in\mathbb N^d$ and a vector $\boldsymbol x\in\mathbb R^d$ we define the partial order
$$\boldsymbol\alpha\leq\boldsymbol\beta\iff\alpha_i\leq\beta_i\ \forall i\in\{1,2,\ldots,d\},$$
the absolute value
$$|\boldsymbol\alpha|\colonequals\alpha_1+\alpha_2+\ldots+\alpha_d,$$
the power
$$\boldsymbol x^{\boldsymbol\alpha}\colonequals x_1^{\alpha_1}x_2^{\alpha_2}\cdots x_d^{\alpha_d}$$
and the high-order partial derivative
$$\partial^{\boldsymbol\alpha}\colonequals\partial_1^{\alpha_1}\partial_2^{\alpha_2}\cdots\partial_d^{\alpha_d}.$$
\end{notation}

\subsection{Fourier analysis}

We adopt the notation and definitions from \cite{koldobsky2005fourier}.

\begin{definition}[{\cite[\S 2.2, \S 2.1]{koldobsky2005fourier}}]
A closed compact set $K\subset\mathbb R^d$ with a non-empty interior is called a \emph{convex body} if it contains the line segment connecting any two of its points. If a convex body $K$ is origin-symmetric, then its \emph{Minkowski functional} defined by
$$\|\boldsymbol x\|_K\colonequals\min\{a\geq 0\mid\boldsymbol x\in aK\}$$
is a norm on $\mathbb R^d$.
\end{definition}

It is easy to see that the Minkowski functional is a homogeneous function of degree $1$ on $\mathbb R^d$ and that
\begin{equation}\label{eq:16}K=\{\boldsymbol x\in\mathbb R^d\mid\|\boldsymbol x\|_K\leq 1\}.\end{equation}
Also, it follows from the definition that the origin is an interior point of every symmetric convex body, so the Minkowski functional is strictly positive outside the origin.

\begin{notation}[{\cite[\S 2.5]{koldobsky2005fourier}}]
We denote by $\mathcal S(\mathbb R^d)$ the space of complex-valued functions $\phi\in C^\infty(\mathbb R^d)$ converging to zero at infinity together with all their derivatives faster than any negative power of $\|\cdot\|_2$. Elements of the space $\mathcal S(\mathbb R^d)$ will be called \emph{test functions}. As usual, we denote by $\mathcal S(\mathbb R^d)'$ the space of continuous linear functionals on $\mathcal S(\mathbb R^d)$, which we call \emph{distributions} over $\mathcal S(\mathbb R^d)$.
\end{notation}

\begin{definition}[{\cite[\S 2.5]{koldobsky2005fourier}}]
We define the \emph{Fourier transform} of a function $\phi\in L_1(\mathbb R^d)$ by
$$\mathcal F\phi(\boldsymbol\xi)\colonequals\hat\phi(\boldsymbol\xi)\colonequals\int_{\mathbb R^d}\phi(\boldsymbol x)e^{-i(\boldsymbol x,\boldsymbol\xi)}\;\mathrm d\boldsymbol x,\quad\boldsymbol\xi\in\mathbb R^d.$$
Further, we define the action of a complex-valued function $f\in L_1(\mathbb R^d)$ on a test function $\phi$ as
$$\langle f,\phi\rangle\colonequals\int_{\mathbb R^d}f(\boldsymbol x)\phi(\boldsymbol x)\;\mathrm d\boldsymbol x$$
and finally, we define the Fourier transform of a distribution $f$ by
$$\langle\hat f,\phi\rangle=\langle f,\hat\phi\rangle.$$
\end{definition}

For any multi-index $\boldsymbol\alpha\in\mathbb N^d$, the derivative of the order $\boldsymbol\alpha$ of a distribution $f$ is defined by
$$\langle\partial^{\boldsymbol\alpha}f,\phi\rangle=(-1)^{|\boldsymbol\alpha|}\langle f,\partial^{\boldsymbol\alpha}\phi\rangle.$$
The Fourier transform is related to differentiation as follows:
\begin{equation}\label{eq:11}(\partial^{\boldsymbol\alpha}f)^\wedge=i^{|\boldsymbol\alpha|}\boldsymbol x^{\boldsymbol\alpha}f^\wedge.\end{equation}

Denote by $\mathrm{St}(k,\mathbb R^d)$ the \emph{Stiefel manifold} of all orthonormal $k$-frames in $\mathbb R^d$ (i.e., the set of ordered orthonormal $k$-tuples of vectors in $\mathbb R^d$).

\begin{definition}[{\cite[\S 3.5]{koldobsky2005fourier}}]
Let $K\subset\mathbb R^d$ be a convex body. For an orthonormal $k$-frame $\Xi\in \mathrm{St}(k,\mathbb R^d)$ we define the \emph{$(d-k)$-dimensional parallel section function} of $K$ by
$$A_{K,\Xi}(\boldsymbol t)=\mathcal H^{d-k}(K\cap\{\langle\Xi\rangle^\perp+t_1\boldsymbol\xi_1+t_2\boldsymbol\xi_2+\ldots+t_k\boldsymbol\xi_k\}),\quad\boldsymbol t\in\mathbb R^k.$$
\end{definition}

The following result expresses the derivatives of the parallel section function $A_{K,\Xi}$ in terms of the Fourier transform of powers of the Minkowski functional.

\begin{lemma}[{\cite[Theorem~3.26]{koldobsky2005fourier}}]\thlabel{lem:01}
Let $K$ be an infinitely smooth origin-symmetric convex body in $\mathbb R^d$, $1\leq k<d$. Then for every orthonormal $k$-frame $\Xi$ in $\mathbb R^d$ and every $s\in\mathbb N$, $s\neq(d-k)/2$,
$$\Delta^sA_{K,\Xi}(\boldsymbol 0)=\frac{(-1)^s}{2^k\pi^k(d-2s-k)}\int_{\mathbb S^{d-1}\cap\langle\Xi\rangle}(\|\cdot\|_K^{-d+2s+k})^\wedge(\boldsymbol\theta)\;\mathrm d\boldsymbol\theta,$$
where $\Delta\colonequals\sum_{i=1}^d\partial^2/\partial x_i^2$ is the Laplace operator on $\mathbb R^d$.
\end{lemma}

\subsection{Field theory}

We adopt the notation and definitions from \cite{roman2005field}. Let $F[x]$ denote the ring of polynomials in a single variable $x$ over a field $F$.

\begin{definition}[{\cite[\S 1.4]{roman2005field}}]
If a polynomial $f(x)\in F[x]$ factors into linear factors
$$f(x)=a(x-\zeta_1)(x-\zeta_2)\cdots(x-\zeta_n)$$
in an extension field $E$, that is, if $\zeta_1,\zeta_2,\ldots,\zeta_n\in E$, we say that $f(x)$ \emph{splits} in $E$.
\end{definition}

\begin{definition}[{\cite[\S 1.4]{roman2005field}}]
Let $\mathcal F=\{f_i(x)\}_{i\in I}$ be family of polynomials over a field $F$. A \emph{splitting field} for $\mathcal F$ is the smallest extension field $E$ of $F$ such that each $f_i(x)\in\mathcal F$ splits over $E$.
\end{definition}

\begin{theorem}[{\cite[Theorem~1.4.1]{roman2005field}}]
Every finite family of polynomials over a field $F$ has a splitting field.
\end{theorem}

\begin{definition}[{\cite[\S 1.5]{roman2005field}}]
Let $E/F$ be a field extension. An element $\zeta\in E$ is said to be \emph{algebraic} over $F$ if $\zeta$ is a root of some polynomial over $F$. An element that is not algebraic over $F$ is said to be \emph{transcendental} over $F$.
\end{definition}

\begin{definition}[{\cite[\S 1.5]{roman2005field}}]
If $\zeta$ is algebraic over $F$, the set of all polynomials with a root at $\zeta$
$$\mathcal I_\zeta\colonequals\{f(x)\in F[x]\mid f(\zeta)=0\}$$
is a non-zero ideal in $F[x]$ and is therefore generated by a unique monic polynomial, called the \emph{minimal polynomial} of $\zeta$ over $F$ and denoted by $\mu_\zeta(x)$.
\end{definition}

The following theorem characterizes minimal polynomials in a variety of useful ways.

\begin{theorem}[{\cite[Theorem~1.5.1]{roman2005field}}]\thlabel{thm:03}
Let $E/F$ be a field extension and let $\zeta\in E$ be algebraic over $F$. Then among all polynomials in $F[x]$, the minimal polynomial $\mu_\zeta(x)$ is:
\begin{enumerate}
\item the unique monic irreducible polynomial $\mu(x)$ for which $\mu(\zeta)=0$;
\item the unique monic polynomial $\mu(x)$ of smallest degree for which $\mu(\zeta)=0$;
\item the unique monic polynomial $\mu(x)$ with the property that for $f(x)\in F[x]$, $f(\zeta)=0$ if and only if $\mu(x)\mid f(x)$.
\end{enumerate}
In other words, $\mu_\zeta(x)$ is the unique monic generator of the ideal $\mathcal I_\zeta$.
\end{theorem}

\subsection{Valued fields}

We adopt the notation and definitions from \cite[\S 2]{engler2005valued}.

\begin{definition}
Let $F$ be a field. A \emph{valuation} on $F$ is a map $v:F\to\mathbb R\cup\{\infty\}$ satisfying the following axioms for all $x,y\in F$:
\begin{enumerate}
\item $v(x)=\infty\iff x=0$;
\item $v(xy)=v(x)+v(y)$;
\item $v(x+y)\geq\min(v(x),v(y))$.
\end{enumerate}
As a consequence, we obtain for all $x,y\in F$:
\begin{enumerate}[resume]
\item $v(1)=0$;
\item $v(x^{-1})=-v(x)$;
\item $v(-x)=v(x)$;
\item $v(x)<v(y)\implies v(x+y)=v(x)$.
\end{enumerate}
\end{definition}

An example of a non-trivial valuation is the $p$-adic valuation on the rational function field $F(\boldsymbol x)$, where $p$ is any irreducible polynomial from $F[\boldsymbol x]$, $F$ being an arbitrary field.

\begin{definition}
Let $F$ be a field. For every irreducible polynomial $p\in F[\boldsymbol x]$, the \emph{$p$-adic valuation} on the rational function field $F(\boldsymbol x)$ is defined by
$$v_p\left(p^\nu\frac{f}{g}\right)=\nu,$$
where $\nu\in\mathbb Z$ and $f,g\in F[x]\setminus\{0\}$ are not divisible by $p$.
\end{definition}

Note that $v_p$ restricted to $F$ is trivial. There is one more interesting valuation on $F(\boldsymbol x)$, trivial on $F$.

\begin{definition}
Let $F$ be a field. The \emph{degree valuation} on the rational function field $F(\boldsymbol x)$ is defined by
$$v_\infty\left(\frac{f}{g}\right)=\deg g-\deg f,$$
where $f,g\in F[\boldsymbol x]\setminus\{0\}$.
\end{definition}

Interestingly enough, there are no valuations on $F[\boldsymbol x]$ other than the ones just mentioned, assuming triviality on $F$ (cf. \cite[Theorem~2.1.4]{engler2005valued}).

\begin{definition}
Let $F$ be a field. A subring $\mathcal O$ of $F$ satisfying $x\in\mathcal O$ or $x^{-1}\in\mathcal O$ for all $x\in F\setminus\{0\}$ is called a \emph{valuation ring} of $F$.
\end{definition}

The following is a direct consequence of Chevalley's Theorem \cite[Theorem~3.1.1]{engler2005valued}:

\begin{theorem}[{\cite[Theorem~3.1.2]{engler2005valued}}]
Let $F_2/F_1$ be a field extension and let $\mathcal O_1\subseteq F_1$ be a valuation ring. Then there is an extension $\mathcal O_2$ of $\mathcal O_1$ in $F_2$.
\end{theorem}

In particular, it means that any valuation $v$ on a field $F$ always admits at least one extension to every field $E$ containing $F$.

\section{Proof of main theorem}

Let us begin with the following regularity lemma:

\begin{lemma}\thlabel{lem:03}
Let $K\subseteq\mathbb R^d$ be a convex body with an infinitely smooth boundary $\partial K$. If the $k$-dimensional isotropic volume function of $K$ satisfies \eqref{eq:01} with $a_1,a_2,\ldots,a_n$ being linearly independent, then $b_1,b_2,\ldots,b_n$ are infinitely smooth in some neighborhood of $t=0$.
\end{lemma}

\begin{proof}
Since $a_1,a_2,\ldots,a_n$ are linearly independent, there exist $H_1,H_2,\ldots,H_n\in\mathrm{Gr}(k,\mathbb R^d)$ such that the alternant matrix
$$A\colonequals\begin{pmatrix}a_1(H_1)&a_2(H_1)&\cdots&a_n(H_1)\\a_1(H_2)&a_2(H_2)&\cdots&a_n(H_2)\\\vdots&\vdots&\ddots&\vdots\\a_1(H_n)&a_2(H_n)&\cdots&a_n(H_n)\end{pmatrix}$$
is invertible. By definition, for every $t\in[0,\infty)$ we have $\boldsymbol v(t)=A\cdot\boldsymbol b(t)$, where
$$\boldsymbol v(t)\colonequals\begin{pmatrix}V_{K,H_1}(t)\\V_{K,H_2}(t)\\\vdots\\V_{K,H_n}(t)\end{pmatrix},\quad\boldsymbol b(t)\colonequals\begin{pmatrix}b_1(t)\\b_2(t)\\\vdots\\b_n(t)\end{pmatrix}.$$
Now, it follows that $\boldsymbol b(t)=A^{-1}\cdot\boldsymbol v(t)$ is infinitely smooth in some neighborhood of $t=0$ because so is $\boldsymbol v(t)$. This concludes the proof.
\end{proof}

Let us also rephrase \thref{lem:01} in an equivalent, coordinate-free way:

\begin{proposition}\thlabel{cor:01}
Let $K$ be an infinitely smooth origin-symmetric convex body in $\mathbb R^d$, $1\leq k<d$. Then for every $k$-dimensional linear subspace $H\in\mathrm{Gr}(k,\mathbb R^d)$ and every $s\in\mathbb N$, $s\neq(d-k)/2$,
$$V_{K,H}^{(2s+k)}(0)=C(d,s,k)\int_{\mathbb S^{d-1}\cap H}(\|\cdot\|_K^{-d+2s+k})^\wedge(\boldsymbol\theta)\;\mathrm d\boldsymbol\theta,$$
where $C(d,s,k)$ is a non-zero constant.
\end{proposition}

\begin{proof}
Let $\Xi\in \mathrm{St}(k,\mathbb R^d)$ be an orthonormal basis of $H$. Clearly, we have
$$V_{K,H}(t)=\int_{\mathbb B^k(t)}A_{K,\Xi}(\boldsymbol u)\;\mathrm d\boldsymbol u,$$
so by \cite[Theorem~3]{Ovall2016TheLA} the $k$-dimensional isotropic volume function admits the series expansion of the form
$$V_{K,H}(t)=\omega_k\sum_{i=0}^s\frac{\Delta^iA_{K,\Xi}(\boldsymbol 0)}{2^ii!\prod_{j=1}^i(2j+k)}t^{2i+k}+o(t^{2s+k}),$$
where $\omega_k$ denotes the volume of the unit ball in $k$ dimensions. In particular, we get
$$V_{K,H}^{(2s+k)}(0)=\omega_k\frac{\Delta^sA_{K,\Xi}(\boldsymbol 0)}{2^ss!\prod_{j=1}^s(2j+k)}(2s+k)!,$$
which further by \thref{lem:01} equals
$$\omega_k\frac{1}{2^ss!\prod_{j=1}^s(2j+k)}(2s+k)!\frac{(-1)^s}{2^k\pi^k(d-2s-k)}\int_{\mathbb S^{d-1}\cap H}(\|\cdot\|_K^{-d+2s+k})^\wedge(\boldsymbol\theta)\;\mathrm d\boldsymbol\theta$$
for every $s\in\mathbb N$, $s\neq(d-k)/2$. This concludes the proof.
\end{proof}

Finally, we are ready to prove the main theorem.

\begin{proof}[Proof of \thref{thm:01}]
The proof will consist of three clearly outlined parts. Firstly, using simple linear algebra, we will reduce the problem to solving an abstract system of polynomial equations. Secondly, using more sophisticated tools of valuation theory, we will eventually characterize its solutions. Finally, we will check the solutions by plugging them into the original problem.

\subsection{Constructing the system of polynomial equations}

Suppose that $K$ is locally $k$-separably integrable. Without loss of generality, we may assume that the functions $a_1,a_2,\ldots,a_n$ are linearly independent. In light of \thref{lem:03}, differentiating \eqref{eq:01} with respect to $t$ yields
\begin{equation}\label{eq:02}V_{K,H}^{(2s+k)}(0)=\sum_{i=1}^na_i(H)b_i^{(2s+k)}(0)\end{equation}
for every $s\in\mathbb N$. Observe that the right-hand sides of \eqref{eq:02} span a finite-dimensional subspace of $C(\mathrm{Gr}(k,\mathbb R^d),\mathbb R)$ of dimension not greater than $n$. Indeed, they are linear combinations of a finite set of functions $a_1,a_2,\ldots,a_n$. Hence also the left-hand sides of \eqref{eq:02} for all $s\in\mathbb N$ span a finite-dimensional subspace of $C(\mathrm{Gr}(k,\mathbb R^d),\mathbb R)$. It follows that for every $s\in\mathbb N$ there exist scalars $c_{s,0},c_{s,1},\ldots,c_{s,n}$, not all zero, such that
\begin{equation}\label{eq:13}\sum_{i=0}^nc_{s,i}V_{K,H}^{(2s+2i+k)}(0)=0.\end{equation}
By virtue of \thref{cor:01}, for every $s\geq\lceil d/2\rceil$ this reads
$$\int_{\mathbb S^{d-1}\cap H}\sum_{i=0}^n\tilde c_{s,i}(-1)^i(\|\cdot\|_K^{-d+2s+2i+k})^\wedge(\boldsymbol\theta)\;\mathrm d\boldsymbol\theta=0,$$
where
$$\tilde c_{s,i}\colonequals(-1)^ic_{s,i}C(d,s+i,k).$$
It means precisely that the $k$-dimensional spherical Radon transform (cf. \cite[\S 2.3]{koldobsky2005fourier}) of the integrand is zero for every $H\in\mathrm{Gr}(k,\mathbb R^d)$. Since $K$ is origin-symmetric, the integrand is an even function, whence
\begin{equation}\label{eq:03}\sum_{i=0}^n\tilde c_{s,i}(-1)^i(\|\cdot\|_K^{-d+2s+2i+k})^\wedge(\boldsymbol\theta)=0\end{equation}
for every $\boldsymbol\theta\in\mathbb S^{d-1}$ (cf. \cite[Corollary~3.10]{koldobsky2005fourier}). Further, using the simple fact that
$$(\|\cdot\|_K^{-d+2s+2i+k})^\wedge(\boldsymbol\theta)=t^d(\|t\cdot\|_K^{-d+2s+2i+k})^\wedge(t\boldsymbol\theta)=t^{2s+2i+k}(\|\cdot\|_K^{-d+2s+2i+k})^\wedge(t\boldsymbol\theta)$$
(cf. \cite[Lemma~2.21]{koldobsky2005fourier}) we can rewrite \eqref{eq:03} in the form
$$\sum_{i=0}^n\tilde c_{s,i}(-1)^it^{2s+2i+k}(\|\cdot\|_K^{-d+2s+2i+k})^\wedge(t\boldsymbol\theta)=0.$$
Dividing both sides by $t^{2s+k}$ and using $\|\boldsymbol\theta\|_2=1$ yields
$$\sum_{i=0}^n\tilde c_{s,i}(-1)^i\|t\boldsymbol\theta\|_2^{2i}(\|\cdot\|_K^{-d+2s+2i+k})^\wedge(t\boldsymbol\theta)=0.$$
By the differentiation property of the Fourier transform \eqref{eq:11}, we get
$$\sum_{i=0}^n\tilde c_{s,i}(\Delta^i\|\cdot\|_K^{-d+2s+2i+k})^\wedge(t\boldsymbol\theta)=0,$$
It follows that the Fourier transform of the distribution
\begin{equation}\label{eq:19}\sum_{i=0}^n\tilde c_{s,i}\Delta^i\|\cdot\|_K^{-d+2s+2i+k}\end{equation}
is supported at the origin, in which case it is a polynomial (cf. \cite[\S 7.16]{rudin1991functional}). Denote this polynomial by $P_s$ and observe that it is homogeneous of degree $-d+2s+k$.

\begin{claim}\thlabel{clm:01}
For any multi-index $\boldsymbol\alpha\in\mathbb N^d$, $m\in\mathbb Z$ and any infinitely smooth function $f:\mathbb R^d\to\mathbb R$ we have
$$\partial^{\boldsymbol\alpha}f^m=Q_{\boldsymbol\alpha}^m(\{\partial^{\boldsymbol\beta}f:\boldsymbol\beta\leq\boldsymbol\alpha\})f^{m-|\boldsymbol\alpha|},$$
where $Q_{\boldsymbol\alpha}^m$ is a polynomial depending on $m$ only through its coefficients. Moreover,
$$Q_{\boldsymbol\alpha}^m(\{\partial^{\boldsymbol\beta}f:\boldsymbol\beta\leq\boldsymbol\alpha\})=m^{|\boldsymbol\alpha|}(\nabla f)^{\boldsymbol\alpha}+O(m^{|\boldsymbol\alpha|-1}),$$
where $(\nabla f)^{\boldsymbol\alpha}$ denotes a multi-index power of the vector $\nabla f\in\mathbb R^d$. In particular, we have
$$\Delta^if^m=\tilde Q_i^m(\{\partial^{\boldsymbol\beta}f:|\boldsymbol\beta|\leq 2i\})f^{m-2i}$$
and
$$\tilde Q_i^m(\{\partial^{\boldsymbol\beta}f:|\boldsymbol\beta|\leq 2i\})=m^{2i}\|\nabla f\|_2^{2i}+O(m^{2i-1}),$$
where $\tilde Q_i^m$ is again a polynomial depending on $m$ only through its coefficients and $\|\nabla f\|_2$ denotes the Euclidean norm of the vector $\nabla f\in\mathbb R^d$.\hfill$\blacksquare$
\end{claim}

Since the proof is a tedious but conceptually straightforward induction on $|\boldsymbol\alpha|$, we leave it to the reader.\\

Applying \thref{clm:01} to \eqref{eq:19} yields
$$\sum_{i=0}^n\tilde c_{s,i}\Delta^i\|\cdot\|_K^{-d+2s+2i+k}=\sum_{i=0}^n\tilde c_{s,i}\tilde Q_i^{-d+2s+2i+k}(\{\partial^{\boldsymbol\beta}\|\boldsymbol x\|_K:|\beta|\leq 2i\})\|\boldsymbol x\|_K^{-d+2s+k}.$$
Thus finally for every $s\geq\lceil d/2\rceil$ and $\boldsymbol x\in\mathbb R^d\setminus\{\boldsymbol 0\}$ we obtain
\begin{equation}\label{eq:04}\sum_{i=0}^n\tilde c_{s,i}\tilde Q_i^{-d+2s+2i+k}(\{\partial^{\boldsymbol\beta}\|\boldsymbol x\|_K:|\beta|\leq 2i\})=P_s(\boldsymbol x)\|\boldsymbol x\|_K^{d-2s-k},\end{equation}
where $\tilde c_{s,i}$ are constants, $\tilde Q_i^{-d+2s+2i+k}$ are polynomials that depend on $s$ only through their coefficients and $P_s$ is a homogeneous polynomial of degree $-d+2s+k$.

\begin{claim}\thlabel{clm:04}
All but finitely many polynomials $P_s$, $s\geq\lceil d/2\rceil$, are non-zero unless $K$ is a Euclidean ball. Indeed, suppose that there exists an increasing sequence $(s_j)_{j\in\mathbb N}$ such that $P_{s_j}=0$ for every $j\in\mathbb N$. Then \eqref{eq:04} reads
$$\sum_{i=0}^n\tilde c_{s_j,i}\tilde Q_i^{-d+2s_j+2i+k}(\{\partial^{\boldsymbol\beta}\|\boldsymbol x\|_K:|\beta|\leq 2i\})=0.$$
It follows immediately from \thref{clm:01} that
\begin{equation}\label{eq:05}\sum_{i=0}^n\tilde c_{s_j,i}s_j^{2i}\left(\big\|\nabla\|\boldsymbol x\|_K\big\|_2^{2i}+O(s_j^{-1})\right)=0.\end{equation}
Denote by $\gamma_j\colonequals\left[\tilde c_{s_j,0}:\tilde c_{s_j,1}s_j^{2}:\ldots:\tilde c_{s_j,n}s_j^{2n}\right]\in\mathbb{RP}^n$ the homogeneous vector of coefficients on the left-hand side of \eqref{eq:05}. Since the projective space is compact, after passing to a subsequence, we may assume without loss of generality that $\gamma_j$ converges to some $\gamma_*\in\mathbb{RP}^n$ as $j$ goes to infinity. That being so, the limiting case of \eqref{eq:05} yields
\begin{equation}\label{eq:06}\sum_{i=0}^n\gamma_{*,i}\big\|\nabla\|\boldsymbol x\|_K\big\|_2^{2i}=0.\end{equation}
Now, as $\big\|\nabla\|\boldsymbol x\|_K\big\|_2$ is continuous and satisfies a polynomial equation with constant coefficients, it must itself be constant, which gives rise to an eikonal equation of the form $\big\|\nabla\|\boldsymbol x\|_K\big\|_2=r^{-1}$, $\boldsymbol x\in\mathbb R^d\setminus\{\boldsymbol 0\}$, where $r>0$ (cf. \cite[Proposition~2.1]{10.2996/kmj/1138043545}). However, in this special case, for $\boldsymbol x\in\partial K$ we may simply write
\begin{equation}\tag*{\cite[(1.39)]{schneider_2013}}\nabla\|\boldsymbol x\|_K=h_K(u_K(\boldsymbol x))^{-1}u_K(\boldsymbol x),\end{equation}
where $u_K:\partial K\to\mathbb S^{d-1}$ is the spherical image map of $K$ \cite[\S 2.5]{schneider_2013} and $h_K:\mathbb R^d\to\mathbb R$ is the support function of $K$ \cite[\S 1.7.1]{schneider_2013}. Hence $h_K(\boldsymbol u)=r$ for all outer unit normal vectors $\boldsymbol u$ in the spherical image of $K$. But since $u_K$ is surjective, it follows that
$$K=\{\boldsymbol x\in\mathbb R^d\mid\langle\boldsymbol x,\boldsymbol u\rangle\leq h_K(\boldsymbol u)\ \text{for all $\boldsymbol u\in\mathbb S^{d-1}$}\}=\{\boldsymbol x\in\mathbb R^d\mid\langle\boldsymbol x,\boldsymbol u\rangle\leq r\ \text{for all $\boldsymbol u\in\mathbb S^{d-1}$}\}=\mathbb B^d(r),$$
whence $K$ is indeed a Euclidean ball.\hfill$\square$
\end{claim}

\begin{remark}
By repeating essentially the same argument, we may see that the left-hand side of \eqref{eq:06} restricted to $\partial K$ is a uniform limit of a certain sequence of homogeneous polynomials. However, since multivariate homogeneous polynomials are dense in the family of continuous even functions on $\partial K$ \cite{Kroo2009143}, this observation does not yield any further constraints.
\end{remark}

Note that the set of monomials (i.e., power products) depending on $\boldsymbol x$ that appear on the left-hand sides of \eqref{eq:04} is finite. In particular, the left-hand sides of \eqref{eq:04} span a finite-dimensional subspace of $C(\mathbb R^d,\mathbb R)$ and thus so do also the right-hand sides. The set of their non-trivial zero linear combinations forms an infinite system of polynomial equations for $\|\boldsymbol x\|_K^{-2}$ with polynomial coefficients. We will investigate it in the next section.

\subsection{Solving the system of polynomial equations}

Denote by $N\in\mathbb N$ the dimension of the subspace of $C(\mathbb R^d,\mathbb R)$ spanned by $\{P_s(\boldsymbol x)\|\boldsymbol x\|_K^{d-2s-k}\}_{s\geq\lceil d/2\rceil}$ and let $\|\boldsymbol x\|_K^{-2}\colonequals\zeta$. Then for any tuple of indices $(s_i)_{0\leq i\leq N}$ there exists a tuple of coefficients $(c_i)_{0\leq i\leq N}$ such that
$$\sum_{i=0}^Nc_iP_{s_i}(\boldsymbol x)\|\boldsymbol x\|_K^{d-2s_i-k}=0.$$
After dividing both sides by $\|\boldsymbol x\|_K^{d-k}$, the above equation reads
\begin{equation}\label{eq:07}\sum_{i=0}^Nc_iP_{s_i}(\boldsymbol x)\zeta^{s_i}=0,\end{equation}
which is an example of a polynomial equation for $\zeta$ with coefficients in the ring $\mathbb R[\boldsymbol x]$ of polynomials in $\boldsymbol x\in\mathbb R^d$. However, it is usually easier to consider polynomial equations over a field, of which we will soon take advantage.\\

In particular, $\zeta$ is algebraic over the field $\mathbb R(\boldsymbol x)$ of rational functions in $\boldsymbol x\in\mathbb R^d$ and thus it has a minimal polynomial of the form
$$\mu_\zeta(\lambda)=\mu_{\zeta,0}+\mu_{\zeta,1}\lambda+\ldots+\mu_{\zeta,m-1}\lambda^{m-1}+\lambda^m,$$
where $\mu_{\zeta,i}\in\mathbb R(\boldsymbol x)$ and $m=[\mathbb R(\boldsymbol x,\zeta):\mathbb R(\boldsymbol x)]$ is the degree of a field extension. Recall that the simple algebraic extension $\mathbb R(\boldsymbol x,\zeta)/\mathbb R(\boldsymbol x)$ is a finite-dimensional vector space over $\mathbb R(\boldsymbol x)$. In fact, the set $\mathcal B=\{1,\zeta,\ldots,\zeta^{m-1}\}$ is a vector space basis for $\mathbb R(\boldsymbol x,\zeta)$ over $\mathbb R(\boldsymbol x)$ (cf. \cite[Theorem~2.4.1]{roman2005field}). The multiplication map $T_\zeta:\mathbb R(\boldsymbol x,\zeta)\to\mathbb R(\boldsymbol x,\zeta)$ is an $\mathbb R(\boldsymbol x)$-linear operator on $\mathbb R(\boldsymbol x,\zeta)$ defined by $T_\zeta(\alpha)=\zeta\alpha$. The matrix of $T_\zeta$ with respect to the ordered basis $\mathcal B$ has the form
$$[T_\zeta]_{\mathcal B}=\begin{pmatrix}&&&&-\mu_{\zeta,0}\\1&&&&-\mu_{\zeta,1}\\&1&&&-\mu_{\zeta,2}\\&&\ddots&&\vdots\\&&&1&-\mu_{\zeta,m-1}\end{pmatrix}$$
and the characteristic polynomial of $[T_\zeta]_{\mathcal B}$ is precisely the minimal polynomial $\mu_\zeta$ (cf. \cite[Theorem~8.1.1]{roman2005field}). The well-known Cayley-Hamilton theorem implies that $\mu_\zeta([T_\zeta]_{\mathcal B})=0$ and therefore $f([T_\zeta]_{\mathcal B})=0$ for any polynomial $f\in\mathcal I_\zeta$, which follows from \thref{thm:03}. In particular, for every linear combination \eqref{eq:07} we have
$$\sum_{i=0}^Nc_iP_{s_i}[T_\zeta]_{\mathcal B}^{s_i}=0.$$

\begin{lemma}\thlabel{lem:02}
Let $X,Y$ be real vector spaces and let $\{\boldsymbol x_i\}_{i\in I}\subseteq X,\{\boldsymbol y_i\}_{i\in I}\subseteq Y$ be sets of vectors. Then there is a linear map $f:X\to Y$ such that $f(\boldsymbol x_i)=\boldsymbol y_i$ for each $i\in I$ if and only if for all finite subsets $J\subseteq I$ and sets of scalars $\{a_j\}_{j\in J}\subseteq\mathbb R$ such that $\sum_{j\in J}a_j\boldsymbol x_j=\boldsymbol 0$ the equality $\sum_{j\in J}a_j\boldsymbol y_j=\boldsymbol 0$ holds.
\end{lemma}

Since the proof is an easy exercise from linear algebra, we leave it to the reader.\\

Now, it follows from \thref{lem:02} that there exists an $\mathbb R$-linear map $f:C(\mathbb R^d,\mathbb R)\to\mathbb R(\boldsymbol x)^{m\times m}$ such that $f(P_s\zeta^s)=P_s[T_\zeta]_{\mathcal B}^s$ for each $s\geq\lceil d/2\rceil$. In particular, the rank-nullity theorem implies
$$\mathrm{dim}(\mathrm{span}(\{P_s[T_\zeta]_{\mathcal B}^s\}_{s\geq\lceil d/2\rceil}))=\mathrm{dim}(\mathrm{span}(\{f(P_s\zeta^s)\}_{s\geq\lceil d/2\rceil}))\leq\mathrm{dim}(\mathrm{span}(\{P_s\zeta^s\}_{s\geq\lceil d/2\rceil}))=N,$$
whence $\{P_s[T_\zeta]_{\mathcal B}^s\}_{s\geq\lceil d/2\rceil}$ span a finite-dimensional subspace of $\mathbb R(\boldsymbol x)^{m\times m}$. Therefore the problem has been reduced to a question about rational functions, to which we can now apply the theory of valued fields.

\begin{claim}\thlabel{clm:02}
Let $L/K$ be a field extension and let $v:L\to\mathbb R\cup\{\infty\}$ be a valuation on $L$, trivial on $K$. Suppose that $\{x_i\}_{i\in I}\subseteq L$ is a subset of $L$ such that $\{v(x_i)\}_{i\in I}\subseteq\mathbb R$ are pairwise different. Then $\{x_i\}_{i\in I}$ is $K$-linearly independent. For suppose that there exists a finite subset $J\subseteq I$ and a set of scalars $\{a_j\}_{j\in J}\subseteq K$ such that $\sum_{j\in J}a_jx_j=0$. Since $v(a_jx_j)=v(a_j)+v(x_j)=v(x_j)$ are again pairwise different, we have
$$\infty=v(0)=v\left(\sum_{j\in J}a_jx_j\right)=\min_{j\in J}v(x_j),$$
which reads $v(x_j)=\infty$ for all $j\in J$, a contradiction. In particular, if $\{x_i\}_{i\in I}$ span a finite-dimensional $K$-linear space, then the set $\{v(x_i)\}_{i\in I}$ is necessarily finite.\hfill$\square$
\end{claim}

Let $p\in\mathbb R[\boldsymbol x]$ be any irreducible polynomial and denote by $\hat v_p$ some extension of the $p$-adic valuation on $\mathbb R(\boldsymbol x)$ to the splitting field of $\mu_\zeta$, i.e., the smallest extension of $\mathbb R(\boldsymbol x)$ containing all eigenvalues $\lambda_1,\lambda_2,\ldots,\lambda_m$ of $[T_\zeta]_{\mathcal B}$. Now, observe that at least one of the coefficients $\mu_{\zeta,m-i}$, $i=1,2,\ldots,m$ is non-zero. Otherwise, the minimal polynomial $\mu_\zeta$ would be reducible, a contradiction. Let $\mu_{\zeta,m-i}$ be some non-zero coefficient of $\mu_\zeta$. By Vi\`ete's formulas for $\mu_\zeta$ we have
$$\mu_{\zeta,m-i}=(-1)^i\sum_{1\leq\alpha_1<\alpha_2<\ldots<\alpha_i\leq m}\lambda_{\alpha_1}\lambda_{\alpha_2}\cdots\lambda_{\alpha_i}.$$
In particular,
\begin{align*}v_p(\mu_{\zeta,m-i})&=\hat v_p\left((-1)^i\sum_{1\leq\alpha_1<\alpha_2<\ldots<\alpha_i\leq m}\lambda_{\alpha_1}\lambda_{\alpha_2}\cdots\lambda_{\alpha_i}\right)\\&\geq\min_{1\leq\alpha_1<\alpha_2<\ldots<\alpha_i\leq m}\hat v_p(\lambda_{\alpha_1}\lambda_{\alpha_2}\cdots\lambda_{\alpha_i})\\&=\min_{1\leq\alpha_1<\alpha_2<\ldots<\alpha_i\leq m}\hat v_p(\lambda_{\alpha_1})+\hat v_p(\lambda_{\alpha_2})+\ldots+\hat v_p(\lambda_{\alpha_i}),\end{align*}
whence there exists an eigenvalue $\lambda$ of $[T_\zeta]_{\mathcal B}$ such that $\hat v_p(\lambda)\leq v_p(\mu_{\zeta,m-i})/i$. Without loss of generality we may reorder the eigenvalues so that $\hat v_p(\lambda_1),\hat v_p(\lambda_2),\ldots,\hat v_p(\lambda_j)\leq\hat v_p(\lambda)$ and $\hat v_p(\lambda_{j+1}),\hat v_p(\lambda_{j+2}),\ldots,\hat v_p(\lambda_m)>\hat v_p(\lambda)$ for some $1\leq j\leq m$. Denote by
$$\chi_{[T_\zeta]_{\mathcal B}^s}(\lambda)\colonequals\det(\lambda I-[T_\zeta]_{\mathcal B}^s)$$
the characteristic polynomial of $[T_\zeta]_{\mathcal B}^s$. Since the eigenvalues of $[T_\zeta]_{\mathcal B}^s$ are precisely $\lambda_1^s,\lambda_2^s,\ldots,\lambda_m^s$, again by Vi\`ete's formulas for $\chi_{[T_\zeta]_{\mathcal B}^s}$ we have
\begin{equation}\label{eq:08}\chi_{[T_\zeta]_{\mathcal B}^s,m-j}=(-1)^j\sum_{1\leq\alpha_1<\alpha_2<\ldots<\alpha_j\leq m}\lambda_{\alpha_1}^s\lambda_{\alpha_2}^s\cdots\lambda_{\alpha_j}^s,\end{equation}
where $\chi_{[T_\zeta]_{\mathcal B}^s,m-j}$ stands for the coefficient of $\lambda^{m-j}$. This time we know, however, that $\lambda_1^s\lambda_2^s\cdots\lambda_j^s$ attains the smallest valuation among all summands on the right-hand side of \eqref{eq:08}, whence
$$v_p(\chi_{[T_\zeta]_{\mathcal B}^s,m-j})=\hat v_p(\lambda_1^s\lambda_2^s\cdots\lambda_j^s)=s(\hat v_p(\lambda_1)+\hat v_p(\lambda_2)+\ldots+\hat v_p(\lambda_j))\leq jsv_p(\mu_{\zeta,m-i})/i.$$
On the other hand, using the fact that $\chi_{[T_\zeta]_{\mathcal B}^s,m-j}$ may be computed as the sum of all principal minors of $[T_\zeta]_{\mathcal B}^s$ of size $j$ and each of those minors may itself be computed as the sum of products of certain entries of the matrix $[T_\zeta]_{\mathcal B}^s$ of length $j$, it follows that
$$v_p(\chi_{[T_\zeta]_{\mathcal B}^s,m-j})\geq jv_p([T_\zeta]_{\mathcal B}^s),$$
where $v_p(M)$ denotes the minimum of valuations taken over all entries of a matrix $M$. Chaining those two inequalities yields
$$v_p([T_\zeta]_{\mathcal B}^s)\leq sv_p(\mu_{\zeta,m-i})/i$$
and consequently
$$v_p(P_s[T_\zeta]_{\mathcal B}^s)\leq v_p(P_s)+sv_p(\mu_{\zeta,m-i})/i.$$
Finally, by \thref{clm:02} we have $v_p(P_s[T_\zeta]_{\mathcal B}^s)=O(1)$, which implies
\begin{equation}\label{eq:09}v_p(P_s)\geq-sv_p(\mu_{\zeta,m-i})/i+O(1).\end{equation}

\begin{claim}\thlabel{clm:05}
Since $\zeta$ is homogeneous of degree $-2$, the coefficient $\mu_{\zeta,m-i}$ is homogeneous of degree $-2i$ for every $1\leq i\leq m$. Indeed, for every $t\neq 0$ and $x\in\mathbb R^d$ we have
\begin{align*}0&=t^{-2m}\mu_\zeta(\zeta(\boldsymbol x))\\&=t^{-2m}\mu_\zeta(t^2\zeta(t\boldsymbol x))\\&=t^{-2m}\mu_{\zeta,0}(x)+t^{-2m+2}\mu_{\zeta,1}(\boldsymbol x)\zeta(t\boldsymbol x)+\ldots+t^{-2}\mu_{\zeta,m-1}(\boldsymbol x)\zeta(t\boldsymbol x)^{m-1}+\zeta(t\boldsymbol x)^m\\&=t^{-2m}\mu_{\zeta,0}(t^{-1}\boldsymbol y)+t^{-2m+2}\mu_{\zeta,1}(t^{-1}\boldsymbol y)\zeta(\boldsymbol y)+\ldots+t^{-2}\mu_{\zeta,m-1}(t^{-1}\boldsymbol y)\zeta(\boldsymbol y)^{m-1}+\zeta(\boldsymbol y)^m,\end{align*}
where $\boldsymbol y\colonequals t\boldsymbol x$. Recall that $\mu_\zeta$ is the unique monic polynomial of degree $m$ with root at $\zeta$, which implies
$$t^{-2i}\mu_{\zeta,m-i}(t^{-1}\boldsymbol y)=\mu_{\zeta,m-i}(\boldsymbol y)$$
for every $1\leq i\leq m$.\hfill$\square$
\end{claim}

Since $\mathbb R(\boldsymbol x)$ is the field of fractions of a unique factorization domain, there is a system of irreducible polynomials $\mathcal P\subset\mathbb R[\boldsymbol x]$ such that every non-zero element $f\in\mathbb R(\boldsymbol x)$ admits a unique representation
\begin{equation}\label{eq:17}f=u\prod_{p\in\mathcal P}p^{v_p(f)},\end{equation}
where $u\in\mathbb R\setminus\{0\}$ is invertible and the integral exponents $v_p(f)\in\mathbb Z$ are non-zero for only a finite number of elements $p\in\mathcal P$. This representation may be viewed as an analog of the \emph{product formula} for rational numbers. Computing the degree valuation of both sides of \eqref{eq:17} yields
\begin{equation}\label{eq:18}v_\infty(f)=-\sum_{p\in\mathcal P}v_p(f)\deg p.\end{equation}

Denote by $\mathcal P_{m-i}\subset\mathcal P$ the finite subset of all irreducible polynomials $p\in\mathcal P$ such that $v_p(\mu_{\zeta,m-i})\neq 0$. From our considerations so far, it follows that
\begin{align*}-d+2s+k&=-v_\infty(P_s)\\\overset{\eqref{eq:18}}&{=}\sum_{p\in\mathcal P}v_p(P_s)\deg p\\&\geq\sum_{p\in\mathcal P_{m-i}}v_p(P_s)\deg p\\\overset{\eqref{eq:09}}&{\geq}\sum_{p\in\mathcal P_{m-i}}-sv_p(\mu_{\zeta,m-i})\deg p/i+O(1)\\\overset{\eqref{eq:18}}&{=}sv_\infty(\mu_{\zeta,m-i})/i+O(1)\\\overset{\text{\thref{clm:05}}}&{=}2s+O(1)\end{align*}
for every non-zero $\mu_{\zeta,m-i}$. Hence all the inequalities used above actually must be equalities up to some bounded error term. Thus
\begin{equation}\label{eq:10}v_p(P_s)=-sv_p(\mu_{\zeta,m-i})/i+O(1)\end{equation}
for every $p\in\mathcal P$. Now, observe that $\mu_{\zeta,0}\neq 0$, because otherwise $\mu_\zeta$ would be reducible. That being so, for every $p\in\mathcal P$ and every non-zero $\mu_{\zeta,m-i}$ we have
$$-sv_p(\mu_{\zeta,m-i})/i+O(1)=v_p(P_s)=-sv_p(\mu_{\zeta,0})/m+O(1),$$
whence asymptotically (as $s$ goes to infinity) we get
$$v_p(\mu_{\zeta,m-i})/i=v_p(\mu_{\zeta,0})/m.$$
By the product formula \eqref{eq:17}, it means that $\mu_{\zeta,m-i}$ and $\mu_{\zeta,0}^{i/m}$ are associated, i.e.,
$$\mu_{\zeta,m-i}=u_{m-i}\mu_{\zeta,0}^{i/m}$$
for some unit $u_{m-i}\in\mathbb R\setminus\{0\}$. On the other hand, if $\mu_{\zeta,m-i}=0$, the same equality holds if we simply put $u_{m-i}\colonequals 0$. Thus
\begin{align*}0&=\mu_{\zeta,0}^{-1}\mu_\zeta(\zeta)\\&=u_0+u_1\mu_{\zeta,0}^{-1/m}\zeta+\ldots+u_{m-1}\mu_{\zeta,0}^{-(m-1)/m}\zeta^{m-1}+\mu_{\zeta,0}^{-1}\zeta^m\\&=u_0+u_1\left(\mu_{\zeta,0}^{-1/m}\zeta\right)+\ldots+u_{m-1}\left(\mu_{\zeta,0}^{-1/m}\zeta\right)^{m-1}+\left(\mu_{\zeta,0}^{-1/m}\zeta\right)^m\end{align*}
is a polynomial equation with constant coefficients satisfied by $\mu_{\zeta,0}^{-1/m}\zeta$, which therefore must itself be constant, equal to some $r\in\mathbb R$. Finally, since $v_p(P_s)\geq 0$ for every $p\in\mathcal P$, it follows from \eqref{eq:10} that $v_p(\mu_{\zeta,0})\leq 0$, which means that
\begin{equation}\label{eq:12}\|\boldsymbol x\|_K=\zeta^{-1/2}=(r^m\mu_{\zeta,0})^{-1/(2m)}\end{equation}
is a root of order $2m$ of some homogeneous polynomial $(r^m\mu_{\zeta,0})^{-1}$ of degree $2m$.

\subsection{Filtering out the incidental solutions}

Although \eqref{eq:12} already imposes a rigid structure on the Minkowski functional $\|\boldsymbol x\|_K$, it does not solve the problem immediately. Indeed, this condition is satisfied, e.g., when $K$ is the unit ball in $\ell^d_{2m}$. Moreover, for every Minkowski functional satisfying \eqref{eq:12} we can easily find a sequence of polynomials $P_s$ such that $\{P_s(\boldsymbol x)\|\boldsymbol x\|_K^{d-2s-k}\}_{k\in\mathbb N}$ span a finite-dimensional subspace of $C(\mathbb R^d,\mathbb R)$. Therefore we need to go back to the very beginning of our argument to get the desired contradiction.\\

Slightly abusing the notation, let $\|\boldsymbol x\|_K\colonequals\zeta$ with $\zeta^{2m}\colonequals h(\boldsymbol x)$ being a homogeneous polynomial of degree $2m$. Without loss of generality, we may assume that $h$ is not a perfect power. Then $A\colonequals\mathbb R(\boldsymbol x,\zeta)$ may be viewed as a graded algebra
$$A=\bigoplus_{i\in C_{2m}}A_i,\quad A_i\colonequals\mathbb R(\boldsymbol x)\zeta^i,$$
where the index set is the cyclic group $C_{2m}$.

\begin{claim}\thlabel{clm:03}
The Laplace operator $\Delta$ defines a graded endomorphism of $A$, i.e., for every function $f\in A_i$, $i\in C_{2m}$, we have $\Delta f\in A_i$. Indeed, for $f\colonequals gh^\nu$, where $g\in\mathbb R(\boldsymbol x)$ and $\nu\colonequals\frac{i}{2m}$, we have
\begin{align}\nonumber\Delta f&=\Delta gh^\nu+2\nu\nabla g\cdot\nabla hh^{\nu-1}+\nu g\Delta hh^{\nu-1}+\nu(\nu-1)g\nabla h\cdot\nabla hh^{\nu-2}\\\label{eq:14}&=\left(\Delta g+2\nu\nabla g\cdot\nabla hh^{-1}+\nu g\Delta hh^{-1}+\nu(\nu-1)g\nabla h\cdot\nabla hh^{-2}\right)h^\nu.\end{align}
Clearly the expression in parentheses is again an element of $\mathbb R(\boldsymbol x)$.\hfill$\square$
\end{claim}

\begin{remark}
Furthermore, it follows from \eqref{eq:14} that $v_p(\Delta f)=v_p(f)-2$ for any generic irreducible polynomial $p\in\mathbb R[\boldsymbol x]$, unless $v_p(f)\in\{0,1\}$. Indeed, for $f\colonequals gp^\nu$, where $g\in\mathbb R(\boldsymbol x)$ is not divisible by $p$ and $\nu\in\mathbb Z$, we have
$$\Delta f=\left(\Delta gp^2+2\nu\nabla g\cdot\nabla pp+\nu g\Delta pp+\nu(\nu-1)g\nabla p\cdot\nabla p\right)p^{\nu-2}.$$
Now, the expression in parentheses is generally not divisible by $p$ unless $p\mid\nabla p\cdot\nabla p$. In other words, the dual variety of the projective hypersurface defined by $p$ is contained in the standard hyperquadric. In particular, for $d\leq 3$ there are no such polynomials $p$ with $\deg p>2$, but already for $d=4$ we have e.g.
\begin{equation}\label{eq:15}\begin{aligned}&4 x_0^8+28 x_0^6 x_1^2+16 x_0^6 x_2^2-20 x_0^6 x_3^2+73 x_0^4 x_1^4+124 x_0^4 x_1^2 x_2^2-90 x_0^4 x_1^2 x_3^2-8 x_0^4 x_2^4+60 x_0^4 x_2^2 x_3^2+33 x_0^4 x_3^4\\
&\quad +84 x_0^2 x_1^6+270 x_0^2 x_1^4 x_2^2-124 x_0^2 x_1^4 x_3^2+180 x_0^2 x_1^2 x_2^4+140 x_0^2 x_1^2 x_2^2 x_3^2+60 x_0^2 x_1^2 x_3^4-48 x_0^2 x_2^6\\
&\quad -124 x_0^2 x_2^4 x_3^2-90 x_0^2 x_2^2 x_3^4-20 x_0^2 x_3^6+36 x_1^8+180 x_1^6 x_2^2-48 x_1^6 x_3^2+297 x_1^4 x_2^4+180 x_1^4 x_2^2 x_3^2-8 x_1^4 x_3^4\\
&\quad +180 x_1^2 x_2^6+270 x_1^2 x_2^4 x_3^2+124 x_1^2 x_2^2 x_3^4+16 x_1^2 x_3^6+36 x_2^8+84 x_2^6 x_3^2+73 x_2^4 x_3^4+28 x_2^2 x_3^6+4 x_3^8,
\end{aligned}\end{equation}
obtained as the dual variety to the complete intersection of two hyperquadrics defined by $x_0^2+x_1^2+x_2^2+x_3^2$ and $x_1^2+2x_2^2+3x_3^2$. The above example is fairly complicated, and unfortunately, this will always be the case as long as we believe in the celebrated Hartshorne's conjecture (cf. \cite{Ein1986}). Nevertheless, we showed that $v_p(\Delta f)=v_p(f)-2$ unless $p\mid\nabla p\cdot\nabla p$, when $v_p(\Delta f)\geq v_p(f)-1$. The question of whether inequality may be replaced with equality seems to be hard, and likely the answer in that generality will be negative.
\end{remark}

Having said all that, we are ready to finish the proof. Firstly, suppose that $d-k$ is odd. Going back as far as \eqref{eq:13}, for every $s\in\mathbb N$ there exist scalars $c_{s,0},c_{s,1},\ldots,c_{s,n}$, not all zero, such that
$$\sum_{i=0}^nc_{s,i}V_{K,H}^{(2(ms+\frac{d-k+1}{2})+2mi+k)}(0)=0,$$
whence the distribution
$$\sum_{i=0}^n\tilde c_{s,i}\Delta^{mi}\|\boldsymbol x\|_K^{2ms+2mi+1}=\sum_{i=0}^n\tilde c_{s,i}\Delta^{mi}h^{s+i+\frac{1}{2m}}=\sum_{i=0}^n\tilde c_{s,i}\Delta^{mi}[h^{s+i}\zeta]$$
is a polynomial (cf. \eqref{eq:19}), i.e., the element of $A_0$. On the other hand, by \thref{clm:03}, it is also an element of $A_1\neq A_0$, in which case it must be zero. It follows from \thref{clm:04} that $K$ is a Euclidean ball.\\

Secondly, suppose that $d-k$ is even and $m>1$. Then for every $s\in\mathbb N$ there exist scalars $c_{s,0},c_{s,1},\ldots,c_{s,n}$, not all zero, such that
$$\sum_{i=0}^nc_{s,i}V_{K,H}^{(2(ms+\frac{d-k+2}{2})+2mi+k)}(0)=0,$$
whence the distribution
$$\sum_{i=0}^n\tilde c_{s,i}\Delta^{mi}\|\boldsymbol x\|_K^{2ms+2mi+2}=\sum_{i=0}^n\tilde c_{s,i}\Delta^{mi}h^{s+i+\frac{2}{2m}}=\sum_{i=0}^n\tilde c_{s,i}\Delta^{mi}[h^{s+i}\zeta^2]$$
is a polynomial, i.e., the element of $A_0$. On the other hand, this time it is also an element of $A_2\neq A_0$, in which case it must be zero anyway. Again, it follows that $K$ is a Euclidean ball.\\

Finally, suppose that $d-k$ is even and $m=1$. But then $K$ defined by \eqref{eq:16} is a hyperquadric, which concludes the proof.
\end{proof}

\begin{remark}
Note that, in fact, we proved a much stronger theorem. Indeed, we did not take full advantage of assumption \eqref{eq:01} (formulated in terms of the space of germs), but instead, we used only its infinitesimal version \eqref{eq:02} (formulated in terms of the space of jets). The latter, in general, contains less information unless we restrict ourselves to analytic functions when the two coincide. However, in order to avoid further complicating an already complicated assumption and thus overshadowing the main idea, we have deliberately abandoned the formally weaker formulation in favor of a much simpler and more intuitive one.
\end{remark}

\section*{Acknowledgements}

We would like to thank Prof. M. Wojciechowski for offering many valuable comments, including the remark on the superposition theorem of Kolmogorov, and Prof. J. Buczyński for pointing out the counterexample \eqref{eq:15} and his eagerness to discuss the algebraic geometry behind the proof.


\end{document}